\documentclass{amsart}

\usepackage{amssymb}
\usepackage{amsmath}
\usepackage{latexsym} 
\usepackage{amsfonts}
\usepackage{amsthm}
\usepackage{mathrsfs}
\usepackage{verbatim} 
\usepackage[pdftex,colorlinks=true, pdfstartview=FitV, linkcolor=magenta, citecolor=blue, urlcolor=blue]{hyperref}
\usepackage[all,2cell,ps]{xy}
\usepackage[pdftex]{graphicx}
\usepackage{caption}
\usepackage{subcaption}
\usepackage{float}

\newcommand{\ds}{\displaystyle}
\newcommand{\ol}{\overline}
\newcommand{\ul}{\underline}
\newcommand{\vl}{\;\vert\;}

\newcommand{\Z}{\mathbb{Z}}
\newcommand{\Q}{\mathbb{Q}}

\newcommand{\shift}{\mathsf{\Sigma}}
\newcommand{\Reg}{\mathrm{reg}}
 
\newcommand{\Hom}{\mathrm{Hom}}

\newcommand{\pdim}[1]{\mathrm{pd}(#1)}

\newcommand{\CMp}{Cohen-Macaulay}
\newcommand{\BS}{Boij-S\"oderberg\ }

\newcommand{\ls}{\leqslant}
\newcommand{\gs}{\geqslant}

\newcommand{\lst}[2]{#1_1,\dots,#1_{#2}}

\newcommand{\pure}[1]{\pi(#1)}
\newcommand{\cc}{\mathbf c}
\newcommand{\dd}{\mathbf d}
\newcommand{\kk}{\mathbf k}

\newcommand{\goodtimes}{\cdot}
\newcommand{\biggoodtimes}{\prod}

\makeatletter
\def\adots{\mathinner{\mkern1mu\raise\p@
\vbox{\kern7\p@\hbox{.}}\mkern2mu
\raise4\p@\hbox{.}\mkern2mu\raise7\p@\hbox{.}\mkern1mu}}
\makeatother

\newcommand{\defi}[1]{\textsf{#1}} 

\renewcommand{\rm}[1]{\mathrm{#1}}
\renewcommand{\b}{\beta}

\newcommand{\s}{\sigma}
\renewcommand{\t}{\tau}

\theoremstyle{plain} 
\newtheorem{thm}{Theorem}[section]

\newtheorem{lem}[thm]{Lemma}
\newtheorem{prop}[thm]{Proposition}
\newtheorem{cor}[thm]{Corollary}

\newtheorem*{conj*}{Conjecture}

\newtheorem{question}[thm]{Question}

\newtheorem{alg}[thm]{Algorithm}
\newtheorem{rem}[thm]{Remark}

\newtheorem*{cor57}{Corollary~\ref{cor:altdecompCI}}

\theoremstyle{definition}
\newtheorem{defn}[thm]{Definition}
\newtheorem{example}[thm]{Example}  

\theoremstyle{remark}  

\newtheorem*{observ*}{Observation}
\newtheorem{case}{Case}
\newtheorem*{claim*}{Claim}

\newtheoremstyle{editorialNotes}
	{10mm} 
	{10mm}
	{\slshape}
	{-80pt}
	{\bfseries}
	{}
	{10mm}
	{}

\theoremstyle{editorialNotes}

\usepackage[colorinlistoftodos, textwidth=3cm, shadow]{todonotes}

\title[Non-simplicial decompositions]{Non-simplicial decompositions of Betti diagrams of complete intersections}
\author[Gibbons]{Courtney Gibbons}
\thanks{The first and fifth authors were partially funded through the AMS Mathematical Research Communities by a grant from the NSF (DMS-0751449).}

\author[Jeffries]{Jack Jeffries}

\author[Mayes]{Sarah Mayes}

\author[Raicu]{Claudiu Raicu}

\author[Stone]{Branden Stone}
\thanks{The fifth author was also partially funded by the NSF grant, Kansas Partnership for Graduate Fellows in K-12 Education (DGE-0742523).}

\author[White]{Bryan White}
\address{Courtney Gibbons, University of Nebraska-Lincoln, Department of Mathematics, 203 Avery Hall, P.O. Box 880130, Lincoln, NE 68588-0130}
\email{s-cgibbon5@math.unl.edu}
\address{Jack Jeffries, Department of Mathematics, University of Utah, 155 S 1400 E RM 233
Salt Lake City, UT, 84112-0090}
\email{jeffries@math.utah.edu}
\address{Sarah Mayes, University of Michigan, Department of Mathematics, 4848 East Hall, 530 Church Street, Ann Arbor, MI, 48104-1043}
\email{mayess@umich.edu}
\address{Claudiu Raicu, Princeton University, Department of Mathematics, 1009 Fine Hall, Princeton, NJ, 08544-1000}
\email{craicu@princeton.edu}
\address{Branden Stone, Mathematics Program, Bard College, P.O. Box 5000, Annandale-on-Hudson, NY 12504}
\email{bstone@bard.edu}
\address{Bryan White, Department of Mathematics and Statistics, MSC01 1115, 1 University of New Mexico, Albuquerque, New Mexico, 87131-0001}
\email{bcwhite@unm.edu}

\subjclass[2010]{Primary: 13D02; Secondary: 13C99}

\begin{document}

\maketitle

\begin{abstract}
	We investigate decompositions of Betti diagrams over a polynomial ring within the framework of Boij--S\"oderberg theory.  That is, given a Betti diagram, we decompose it into pure diagrams. Relaxing the requirement that the degree sequences in such pure diagrams be totally ordered, we are able to define a multiplication law for Betti diagrams that respects the decomposition and allows us to write a simple expression the decomposition of the Betti diagram of any complete intersection in terms of the degrees of its minimal generators. In the more traditional sense, the decomposition of complete intersections of codimension at most 3 are also computed as given by the totally ordered decomposition algorithm obtained from \cite{ES1}. In higher codimension, obstructions arise that inspire our work on an alternative algorithm. 
\end{abstract}

\section{Introduction}

Algebraists have long accepted that it is useful to discard structure and work with numerical invariants; the new insight arising from Boij--S\"oderberg theory is that focusing on numerical invariants \textit{up to rational multiple} can also yield information about modules.  Thinking of Betti diagrams as integral points on rays in a rational vector space produces a convex polytope with a simplicial structure (as first conjectured in \cite[Conjectures 2.4, 2.10]{BS1}), and this structure leads to Algorithm~\ref{StandardDecomp} for decomposing Betti diagrams into nonnegative rational combinations of pure diagrams that are linearly independent \cite[Decomposition Algorithm]{ES1}.  Recent results harness the power of  Algorithm~\ref{StandardDecomp}; examples include a proof of the Multiplicity Conjecture of Herzog, Huneke, and Srinivasan (see \cite[Section 2.2]{BS1}), finding a polynomial bound on the regularity of an ideal in terms of half its syzygies \cite{McC} and obtaining a structural result for decompositions of Betti diagrams of a class of Gorenstein ideals \cite[Theorem 5.4]{NS}.

However, Algorithm~\ref{StandardDecomp} clashes with some algebraic structures like the tensor product, and the output of the algorithm can be hard to predict even for seemingly simple objects like complete intersections. In this paper, we will consider an alternative decomposition of Betti diagrams where the required ordered chain condition of Algorithm~\ref{StandardDecomp} is relaxed.  The following example compares two decompositions of a Betti diagram; we refer the reader to Section~\ref{sec:def-notation} for some of the necessary background and notation.  Let  $S = \mathbb{Q}[x,y,u,v]$, and consider the complete intersection $M = \mathbb{Q}[x, y, u, v]/(x, y^2, u^4, v^8)$ as an $S$-module. Then, decomposing the Betti table of $M$ over $S$ via Algorithm~\ref{StandardDecomp}, we obtain the sum:
\begin{equation}\label{standardexample}
\begin{aligned}
\beta(M) &= 168 \pure{0,1,3,7,15} + 60 \pure{0,2,3,7,15} + 210 \pure{0,2,5,7,15} \\ 
&+ 30 \pure{0,4,5,7,15} + 60 \pure{0,4,6,7,15} + 240 \pure{0,4,6,11,15} \\ 
&+ 240 \pure{0,4,9,11,15} + 60 \pure{0,8,9,11,15} + 30 \pure{0,8,10,11,15} \\
&+ 210 \pure{0,8,10,13,15} + 60 \pure{0,8,12,13,15} + 168 \pure{0,8,12,14,15}.
\end{aligned}
\end{equation}
In this decomposition, there is no obvious relationship between the coefficients or degree sequences and the degrees of the forms in the ideal defining $M$.  We pursue this theme in Section~\ref{sec:codim4}, where we illustrate that there is not a simple closed formula for such decompositions of complete intersections. However, if we are willing to relax the requirement of 
Algorithm~\ref{StandardDecomp} that the corresponding degree
sequences form an ordered chain, then there is a simple decomposition
of $\beta(M)$ determined by the degrees of the forms in the ideal defining $M$ (see Section~\ref{sec:main}):
\begin{equation}\label{cidecompexample}
\beta(M) = (1 \cdot 2 \cdot 4 \cdot 8) \sum_{\sigma\in \mathrm{Perm}(\{1,2,4,8\})}\pure{0,\sigma(1),\sigma(1)+\sigma(2),\sigma(1)+\sigma(2)+\sigma(4),15}.
\end{equation}

In contrast with the formula from \eqref{standardexample}, it is easy to parametrize the degree sequences that
arise in \eqref{cidecompexample}. Even better, the coefficients are uniform: each coefficient equals the multiplicity
of $M$. On the other hand, equation~\eqref{cidecompexample} involves 24 summands, whereas equation~\eqref{standardexample} only involves
12.
The decomposition in \eqref{cidecompexample} has origins in Section 4 of Boij-S\"oderberg's \cite{BS1}; we prove that
a natural generalization of the formula in \eqref{cidecompexample} holds in general.  We further prove that \eqref{cidecompexample}
follows immediately from a natural multiplication law for Betti diagrams that is induced by
the tensor product of free complexes (see Section~\ref{sec:main}). From this perspective, equation \eqref{cidecompexample} is simply
the expanded version of a product of pure Betti diagrams:
\begin{equation}
\beta(M)=\pure{0,1} \goodtimes \pure{0,2} \goodtimes \pure{0,4} \goodtimes \pure{0,8}.
\end{equation}

In our view, the most interesting aspect of this work is the insight that there are cases where the coefficients in a decomposition behave better when the decomposition does not respect the partial order.  From this perspective, the uniqueness of the decomposition provided by Algorithm~\ref{StandardDecomp} is, at times, a handicap instead of a boon.  Once we adopt this perspective, the multiplication law for pure diagrams and its corresponding formula are actually quite elegant. We obtain the following:

\begin{cor57}
Suppose that $S/I$ is a complete intersection where $I=(f_1, \dots, f_n)$ and $E=(e_1, \dots, e_n)$ such that $f_i$ is of degree $e_i$. Writing $\mathfrak{S}_n$ for the group of permutations of $\{1,\dots,n\}$, one has the following decomposition of $\beta(S/I)$ as a nonnegative rational sum of pure diagrams: 
\[
    \beta(S/I) = e_1\cdot e_2 \cdots e_n \cdot \sum_{\sigma \in \mathfrak{S}_n} \pure{0,e_{\s_1},e_{\s_1}+e_{\s_2},\dots,e_{\s_1}+\cdots+e_{\s_n}}.
\]
\end{cor57} 

This result provides evidence that alternative decompositions of Betti diagrams may be simpler and easier to compute than the decompositions produced by Algorithm~\ref{StandardDecomp}.  When we expand our outlook to include other types of decompositions, several new questions arise.  For example, beyond complete intersections, what modules have a simple decomposition when we no longer require that the diagrams in the decomposition respect the partial order?  Do the diagrams or coefficients appearing in such decompositions convey algebraic information about the modules they decompose? By extending the natural product structure on Betti diagrams to a product on sums of pure diagrams, we endow the Boij-S\"oderberg cone with additional algebraic structure.  In so doing, what extra insight might we glean about the global structure of syzygies?

In Section~\ref{sec:def-notation}, we develop the necessary background and notation.  Section~\ref{sec:low-codim} demonstrates the difficulty in predicting the outcome of Algorithm~\ref{StandardDecomp}, even in low codimension. Even so, a simply closed formula is given for complete intersections of codimension at most three.  We investigate further difficulties in Section~\ref{sec:codim4} for the codimension four case.  In Section~\ref{sec:main} we show how to write the multiplication pure diagrams in terms of pure diagrams and give applications in the complete intersection case, relaxing the requirement that the degree sequences be totally ordered.

\section{Basic definitions and notation}\label{sec:def-notation}

Let $S=\kk[\lst x n]$ be a polynomial ring over a field $\kk$. We view ${S = \oplus_{i=0 \, }^\infty  S_i}$ as a graded ring with the standard grading.    For a graded $S$-module $M$ and any integer $t$, we denote the \defi{twist of  $M$ by $t$} as the module $M(t)$ whose graded pieces are defined by $M(t)_i = M_{i+t}$.  

Given $M$ an $S$-module of finite length, it has a minimal graded free resolution of the form
\[
	\xymatrixrowsep{5mm}
	\xymatrixcolsep{5mm}
	\xymatrix
		{
			0  & M  \ar[l] & \ds \bigoplus_j S(-j)^{\beta_{0,j}} \ar[l] & \ds \bigoplus_j S(-j)^{\beta_{1,j}} \ar[l] & \cdots \ar[l] & \ds \bigoplus_j S(-j)^{\beta_{c,j}} \ar[l] & 0 \ar[l].
		}
\]
The number $c$ is an invariant of $M$ and is called the \defi{projective dimension of $M$}, denoted $\pdim M$.  The integer $\beta_{ij}$ is the number of degree $j$ generators of a basis of the free module in the $i^{\text{th}}$ step of the resolution.  These $\beta_{ij}$ are independent of the choice of minimal free resolution, and we call these invariants the \defi{graded Betti numbers}.  The \defi{Betti diagram} of $M$ is defined to be 
\[
	\beta(M) = \begin{pmatrix}
		\vdots & \vdots & \vdots & \adots & \\
		\beta_{0{-1}} & \beta_{10} & \beta_{21} & \cdots &\\
		\beta_{00} & \beta_{11} & \beta_{22} & \cdots & \\
		\beta_{01} & \beta_{12} & \beta_{23} & \cdots & \\
		\vdots & \vdots & \vdots &\ddots &
	\end{pmatrix}.
\]
In this paper, we consider $\beta(M)$ as an element of  the vector space $V = \oplus_{i=0}^n \oplus_{j \in \Z} \Q$. If $D \in V$, then we say that $D$ is a \defi{diagram}.  

	Viewing $\beta(M)$ as a diagram, we would like to decompose it into a combination of ``pure diagrams''.  We say $\dd \in \Z^{n+1}$ is a \defi{degree sequence} if $d_i < d_{i+1}$ for all $i$ and that $\dd \ls \dd'$ if $d_i \ls d_i'$ for all $i$.  A \defi{chain of degree sequences} is a totally ordered collection $\{\cdots < \dd^0 < \dd^1 < \cdots < \dd^s < \cdots \}$.  We say that $M$ has a \defi{pure resolution} if $\beta(M)$ has at most one non-zero entry in each column.  For example, if $S = \kk [x,y,z]$ and $M = \kk [x,y,z]/(x^2,xy,y^2,xz)$ then 
	\[
		\beta(M) = \begin{pmatrix}
			1 & - & - & - \\
			- & 4 & 4 & 1
		\end{pmatrix}
	\]
gives a pure resolution. If $M$ has a pure resolution, then for each nonnegative integer $i \leq \pdim M$ there exists an integer $d_i$ for which $\beta_{ij}(M) \not= 0$ if and only if $j = d_i$.  In this case we say that $M$ has a \defi{pure resolution of type $\dd = (d_0,d_1, \dots, d_n)$}.  Further, if $\dd$ is a degree sequence then we can construct a diagram $\pure \dd \in V$ by 
\[ 
\pure \dd_{ij} = 
		\begin{cases}
				\prod_{k \not = i}\frac{1}{|d_i - d_k|}, & j = d_i \\
					 0 ,									& j \not = d_i
					 \end{cases}
			.
	\]
For example, if $\dd = (0,2,3,4)$, then 
\[
	\pure \dd = \begin{pmatrix}
		1/24 & - & - & - \\
		- & 1/4 & 1/3 & 1/8
	\end{pmatrix}.
\]

In proving the conjectures of M. Boij and J. S\"oderberg \cite{BS2},  D. Eisenbud and F.O. Schreyer show there is a unique decomposition of Betti tables in terms of $\pure \dd $ \cite{ES1}. 

\begin{thm}[\cite{BS1,ES1}]\label{thm:uniqueDecomp}
	Let $S = \kk [\lst x n]$ and $M$ an $S$-module of finite length.  Then there is a unique chain of degree sequences $\{\dd^0 < \cdots < \dd^s\}$ and a unique set of scalars $a_i \in \Q$ such that 
	\[
		\beta(M) = \sum_{i = 0}^s a_i \pure{\dd^i}.
	\]
\end{thm}

The unique decomposition in Theorem \ref{thm:uniqueDecomp} respects the partial order (see \cite[Definition 2]{BS2}) of the $\dd^i$'s and is obtained by applying the greedy algorithm to a special chain of degree sequences.  As defined in \cite{BS1}, the \defi{maximal shifts} and \defi{minimal shifts} of degree $i$ of a module $M$ are ${\ol d_i (M) = \max \{ j \vl \beta_{ij}(M) \not= 0\} }$ and \linebreak ${\ul d_i (M) = \min \{ j \vl \beta_{ij}(M) \not= 0\} }$, respectively.  If $M$ is \CMp, then the sequences defined by ${\ol \dd = (\ol d_0, \ol d_1, \dots, \ol d_c) }$ and ${\ul \dd = (\ul d_0, \ul d_1, \dots, \ul d_c) }$ are degree sequences.  

\begin{alg}[Totally Ordered Decomposition Algorithm \cite{ES1}] \label{StandardDecomp}
Let  $S$ be $\kk [\lst x n]$ and $M$ be a finitely generated $S$-module of finite length.  Set $\beta = \beta(M)$.
	\begin{enumerate}
		\item[Step 1:] Identify the minimal degree sequence $\ul \dd$ of $\beta$;
		\item[Step 2:] Choose $q > 0 \in \Q$ maximal such that $\beta - q\pure{\ul \dd}$ has non-negative entries;
		\item[Step 3:] Set $\beta = \beta - q\pure{\ul \dd}$;
		\item[Step 4:] Repeat Steps 1-3 until $\beta$ is a pure diagram;
		\item[Step 5:] Write $\beta(M)$ as a sum of the the $q\pure{\ul \dd}$ obtained in the above steps.
	\end{enumerate}		
\end{alg}

We note that our choice of $\pure{\dd}$ differs from the choices used in \cite{BS1} and \cite{ES1}. In \cite{BS1}, they choose the pure diagram with $\beta_{0}=1$; in \cite{ES1}, they choose the smallest possible pure diagram with integral entries. Since the pure diagrams with degree sequence $\dd$ form a one-dimensional vector space, this different choice only affects the coefficients that arise in the algorithm. As we will note in Remark~\ref{goodchoice}, one advantage to our choice is certain uniform formulas.

 	Let $D \in V = \oplus_{i=0}^n \oplus_{j \in \Z} \Q$ be a diagram. Define the \defi{dual} of $D$, denoted $D^*$, via the formula 
	\[
		(D^*)_{ij} = D_{n-i,-j},
	\] 
and define \defi{twist} $D(r)$ via the formula 
	\[
		D(r)_{ij} = D_{i,r+j}.
	\]  
These definitions mimic the functors $\Hom_S(-,S) = -^*$ and $-\otimes S(r)$ for modules; one may check that $\beta(M^*) = \beta(M)^*$ and $\beta(M(r)) = (\beta(M))(r)$.  In particular, if $M$ is a Gorenstein module of finite length, the Betti diagram will be self-dual up to shift by $\Reg(M)$.

\begin{thm}[Symmetric Decomposition \cite{EKS},\cite{NS}]\label{symmetric} 
	Let $S=\kk [x_1,\ldots,x_n]$ and $M$ a Gorenstein $S$-module of finite length.  Then the decomposition of $\beta(M)$ via Algorithm \ref{StandardDecomp} is symmetric; i.e., 
		\[
			\beta(M) = a_0 \pure{\dd^0} + a_1 \pure{\dd^1} + \cdots + a_1 \pure{\dd^1}^*(r) + a_0 \pure{\dd^0}^*(r)
		\] 
 	where $r = \Reg(M)$.
\end{thm}

\section{Complete Intersections in Low Codimension}\label{sec:low-codim}

Let $S = \kk [\lst x d]$ be a polynomial ring over a field $\kk$ and let $\lst f d$ be a homogeneous regular sequence.  If $I = (\lst f d)$, then the ring $S/I$ is called a \defi{complete intersection}.  The Koszul complex on $\lst f d$ provides a minimal resolution of $S/I$ over $S$, so the Betti table, hence also the \BS decomposition, of $S/I$ is completely determined by the degrees $e_i$, i.e., the \defi{type} $(e_1,\dots,e_d)$ of the complete intersection $S/I$. In this section we investigate the following question concerning the decomposition of Betti tables arising from complete intersections.

\begin{question}\label{quest:ci-decompose}
	Let $S = \kk [\lst x d]$ and $I = (\lst f d)$ be an ideal of $S$ generated by a homogeneous regular sequence with $\deg(f_i) = e_i$.  How does the Betti table decomposition from Algorithm~\ref{StandardDecomp} depend on the degrees $e_i$? Can we describe this relationship by a simple formula?
\end{question}

 In codimension at most three, we are able to answer this, as the decomposition behaves uniformly. When the codimension is larger than three, a fundamental complication arises in the bookkeeping, leaving Question~\ref{quest:ci-decompose} unanswered.

\begin{prop}\label{prop:structure}
	Let $S = \kk [\lst x d]$ and $I = (\lst f d)$ be an ideal generated by a homogeneous regular sequence with $\deg(f_i) = e_i$ where $e_i \leq e_{i+1}$ for all $i$.  If $d \ls 3$, then the Betti table decomposition obtained from Algorithm~\ref{StandardDecomp} is completely determined by the degrees $\lst e d$.  In particular, we have the following.  
	\begin{enumerate}
		\item[If $d = 1$:] 
		\[ \beta(S/I) = e_1\cdot\pure{0,e_1}. \]
		
		\item[If $d = 2$:] 
		\[
			\beta(S/I) = e_1e_2\cdot\pure{0,e_1,e_1+e_2} + e_1e_2 \cdot\pure{0,e_2,e_1+e_2}.
		\]
	 	
		\item[If $d = 3$:] 
		\begin{align*}
			\beta(S/I) = &\  e_1e_2(e_2+e_3)\cdot\pure{0,e_1,e_1+e_2,e_1+e_2+e_3}\\
			 			 & + e_1e_2(e_3-e_1)\cdot \pure{0,e_2,e_1+e_2,e_1+e_2+e_3}\\
					     & + 2e_1e_2(e_1+e_3-e_2)\cdot \pure{0,e_2,e_1+e_3,e_1+e_2+e_3}\\
					 	 & + e_1e_2(e_3-e_1)\cdot \pure{0,e_3,e_1+e_3,e_1+e_2+e_3}\\
					 	 & + e_1e_2(e_2+e_3)\cdot \pure{0,e_3,e_2+e_3,e_1+e_2+e_3}.
		\end{align*}
	\end{enumerate}
\end{prop}
\begin{proof} 
For codimension 1, we have $S=\kk [x_1]$ and $f_1$ is a nonzero homogeneous element of degree $e_1$.  The Betti table of $S/(f_1)$ is already pure so no decomposition is needed.  Hence we have $\beta(S/(f_1)) = e_1\cdot\pure{0,e_1}$.

For the codimension 2 case, let $S=\kk [x,y]$ and $f, g$ be a homogeneous regular sequence of type $(e_1,e_2)$.  If we apply Algorithm \ref{StandardDecomp} to $S/(f,g)$, we obtain the following decomposition:  
\[
	\beta(S/(f,g)) = e_1 e_2 \cdot \pure{0,a,a+b} + e_1 e_2 \cdot \pure{0,b,a+b}.
\]

In the codimension 3 case, let $B$ denote the sum of pure diagrams appearing in the statement of the theorem. Note that the set of pure diagrams appearing in $B$ forms a chain, and the sum is symmetric in the sense of Theorem~\ref{symmetric}. Namely, 
\[ \pure{0,e_3,e_2+e_3,e_1+e_2+e_3} = \pure{0,e_1,e_1+e_2,e_1+e_2+e_3}^*(\Reg(S/I)); \]
\[ \pure{0,e_3,e_1+e_3,e_1+e_2+e_3} = \pure{0,e_2,e_1+e_2,e_1+e_2+e_3}^*(\Reg(S/I)). \]
Since $S/I$ is Gorenstein, Theorem~\ref{symmetric} applies to show that the \BS decomposition is symmetric. We only need to check that the zeroth and first columns of $\beta(S/I)$ and $B$ are equal, for then the entire tables are equal and the uniqueness of the \BS decomposition shows that the sum $B$ is the result of Algorithm~\ref{StandardDecomp}.
One verifies in column zero:
\begin{align*}
 B_{00}=&\frac{1}{e_1+e_2+e_3}\Big(\frac{e_1 e_2 (e_2+e_3)}{e_1(e_1+e_2)}+\frac{e_1 e_2 (e_3 -e_1)}{e_2 (e_1 + e_2)} \\
&+ \frac{2 e_1 e_2 (e_1 + e_3 - e_2)}{e_2 (e_1 + e_3)} + \frac{ e_1 e_2 (e_3-e_1)}{e_3 (e_1+e_3)} + \frac{ e_1 e_2 (e_2 + e_3)}{e_3 (e_2+e_3)}\Big)\\
&= \frac{1}{e_1+e_2+e_3} \Big( e_2+e_3-e_1 + 2e_1 - \frac{2 e_1 e_2} {e_1 + e_3} + \frac{2 e_1 e_2}{e_1 + e_3}\Big) =1,
\end{align*}
and all other $B_{0j}$ are zero.

In column one, the only nonzero entry of $\pure{0,e_1,e_1+e_2,e_1+e_2+e_3}$ is in row $e_1$, with value
\[\frac{e_1 e_2 (e_2 + e_3)}{e_1 e_2 (e_2 + e_3)} = 1.\]
The diagrams $\pure{0,e_2,e_1+e_2,e_1+e_2+e_3}$ and $\pure{0,e_2,e_1+e_3,e_1+e_2+e_3}$ have their only nonzero entry in column one at row $e_2$. The total contribution there is
\[\frac{e_1e_2(e_3-e_1)} {e_1 e_2 (e_2 + e_3)} + \frac{2e_1e_2(e_1+e_3-e_2)} {e_2 (e_1 + e_3) (e_1 + e_3 - e_2) }= \frac{e_3 - e_1}{e_1 + e_3} + \frac{2 e_1}{e_1 + e_3} = 1.\]
The remaining diagrams have in column one nonzero entry only at row $e_3$. The value there is
\[ \frac{ e_1 e_2 (e_3 -e_1)}{e_1 e_3 (e_1 + e_2) } + \frac{e_1 e_2 (e_2 + e_3)}{e_2 e_3 (e_1 + e_2)} = \frac{e_2 e_3 +e_1 e_3} {e_3 (e_1 + e_2)} =1.\]
Thus, $B$ has nonzero entries $B_{1j}$ only for $j=e_1, e_2, e_3$ with $B_{1j}$ equal to the number of $e_i$ equal to $j$. So, $B$ agrees with $\beta(S/I)$ in columns zero and one, as required.
\end{proof}

\section{Examples in Codimension Four}\label{sec:codim4}

	As seen in Proposition~\ref{prop:structure}, for codimension up to three, the decomposition via Algorithm~\ref{StandardDecomp} of a complete intersection behaves uniformly.  That is, the coefficients can always be determined by a single formula in terms of the degrees. This is not the case in codimension four or greater. The main point we want to make in this section is that one cannot express the Boij-S\"oderberg decomposition of a codimension four (or greater) complete intersection by a simple formula in terms of the degrees. To help us articulate this point, we create the following definition.

\begin{defn} 
Given a diagram $D \in V$, its \defi{elimination table} has as its $(i,j)^{\text{th}}$ entry the integer $k$ such that the $k^{\text{th}}$ iteration of Algorithm~\ref{StandardDecomp} applied to $D$ is the first iteration such that the $(i,j)^{\text{th}}$ entry of $D$ becomes zero.  
\end{defn}

The elimination table is a means of recording the elimination order of the row and column position according to Algorithm~\ref{StandardDecomp}.  Given any diagram $D \in V$, the sequence of pure diagrams appearing in the Boij-S\"oderberg decomposition of $D$ can be obtained recursively from the elimination table. Indeed, in Algorithm~\ref{StandardDecomp}, the degree sequence of the pure diagram corresponding to the $t^{\text{th}}$ iteration is given by the sequence of least degrees in each column after $t-1$ eliminations. We may read this in the elimination table as entries of least degree in each column (i.e., highest up on the page) among those with value at least $t$.

	One consequence of Proposition~\ref{prop:structure} is that there is only one form of the elimination table for any complete intersection of codimension at most three.  This form is defined by the degree sequences in the simple closed formulas. As we see below, this is not the case for codimension four since we are able to find multiple degree sequences whose forms are incompatible.

\begin{example} Let $S = \kk [x,y,u,v]$.  We work with the Betti diagrams of $S$ modulo each of the following ideals: $I_1 = (x^3,y^4,u^5,v^7)$, $I_2 = (x,y^2,u^4,v^8)$, and $I_3 = (x^4,y^5,u^7,v^9)$.  In each case, every nonzero Betti number is $1$, so we only display the elimination tables:

\begin{figure}[H]
        \centering
        \begin{subfigure}[t]{.3\textwidth}
				\[
					\begin{pmatrix}
						12&\text{.}&\text{.}&\text{.}&\text{.}\\
						\text{.}&\text{.}&\text{.}&\text{.}&\text{.}\\
						\text{.}&2&\text{.}&\text{.}&\text{.}\\
						\text{.}&5&\text{.}&\text{.}&\text{.}\\
						\text{.}&8&\text{.}&\text{.}&\text{.}\\
						\text{.}&\text{.}&1&\text{.}&\text{.}\\
						\text{.}&12&3&\text{.}&\text{.}\\
						\text{.}&\text{.}&6&\text{.}&\text{.}\\
						\text{.}&\text{.}&9&\text{.}&\text{.}\\
						\text{.}&\text{.}&11&4&\text{.}\\
						\text{.}&\text{.}&12&\text{.}&\text{.}\\
						\text{.}&\text{.}&\text{.}&7&\text{.}\\
						\text{.}&\text{.}&\text{.}&10&\text{.}\\
						\text{.}&\text{.}&\text{.}&12&\text{.}\\
						\text{.}&\text{.}&\text{.}&\text{.}&\text{.}\\
						\text{.}&\text{.}&\text{.}&\text{.}&12\\
					\end{pmatrix}
				\]
         \end{subfigure}\hfill
        \begin{subfigure}[t]{.3\textwidth}
				\[
					\begin{pmatrix}
						12&1&\text{.}&\text{.}&\text{.}\\
      \text{.}&3&2&\text{.}&\text{.}\\
      \text{.}&\text{.}&\text{.}&\text{.}&\text{.}\\
     \text{.}&7&4&\text{.}&\text{.}\\
     \text{.}&\text{.}&6&5&\text{.}\\
     \text{.}&\text{.}&\text{.}&\text{.}&\text{.}\\
     \text{.}&\text{.}&\text{.}&\text{.}&\text{.}\\
     \text{.}&12&8&\text{.}&\text{.}\\
      \text{.}&\text{.}&10&9&\text{.}\\
      \text{.}&\text{.}&\text{.}&\text{.}&\text{.}\\
      \text{.}&\text{.}&12&11&\text{.}\\
      \text{.}&\text{.}&\text{.}&12&12\\
					\end{pmatrix}
				\]
         \end{subfigure}\hfill
         \begin{subfigure}[t]{.3\textwidth}
				\[
					\begin{pmatrix}	
						8&\text{.}&\text{.}&\text{.}&\text{.}\\
						\text{.}&\text{.}&\text{.}&\text{.}&\text{.}\\
						\text{.}&\text{.}&\text{.}&\text{.}&\text{.}\\
						\text{.}&1&\text{.}&\text{.}&\text{.}\\
						\text{.}&3&\text{.}&\text{.}&\text{.}\\
						\text{.}&\text{.}&\text{.}&\text{.}&\text{.}\\
						\text{.}&6&\text{.}&\text{.}&\text{.}\\
						\text{.}&\text{.}&1&\text{.}&\text{.}\\
						\text{.}&8&\text{.}&\text{.}&\text{.}\\
						\text{.}&\text{.}&2&\text{.}&\text{.}\\
						\text{.}&\text{.}&4&\text{.}&\text{.}\\
						\text{.}&\text{.}&6&\text{.}&\text{.}\\
						\text{.}&\text{.}&7&\text{.}&\text{.}\\
						\text{.}&\text{.}&\text{.}&2&\text{.}\\
						\text{.}&\text{.}&8&\text{.}&\text{.}\\
						\text{.}&\text{.}&\text{.}&5&\text{.}\\
						\text{.}&\text{.}&\text{.}&\text{.}&\text{.}\\
						\text{.}&\text{.}&\text{.}&7&\text{.}\\
						\text{.}&\text{.}&\text{.}&8&\text{.}\\
						\text{.}&\text{.}&\text{.}&\text{.}&\text{.}\\
						\text{.}&\text{.}&\text{.}&\text{.}&\text{.}\\
						\text{.}&\text{.}&\text{.}&\text{.}&8\\
					\end{pmatrix}
				\]
         \end{subfigure}
         \caption{Elimination tables of $S/I_1$, $S/I_2$, and $S/I_3$ respectively.}
\end{figure}
For $S/I_1$, the first elimination in this example occurs in column 2, and the subsequent eliminations switch between columns. For $S/I_2$, the first elimination occurs in column 1, and again the subsequent eliminations switch between columns without any multiple eliminations (besides the final one).  However, applying Algorithm~\ref{StandardDecomp} to $S/I_3$, we see multiple elimination on the first, second, sixth, seventh, and final iteration.  Notice that since only 8 eliminations occur, the Boij-S\"oderberg decomposition of $S/I_3$ has only 8 distinct pure diagrams occurring. 

\end{example}

This example displays the difficulty in identifying a chain of degree sequences, and thus a formula for the coefficients in the codimension four case.  Indeed,~one can verify by hand that for $I$ the complete intersection $(x^a,y^b,u^c,v^d)$, with ${a<\!b<\!c<\!d}$, then the first elimination of $S/I$ occurs in column 1 if ${a(b+2c+d)>c(c+d)}$, in column 2 if ${a(b+2c+d)<c(c+d)}$, and multiple elimination occurs in the first step if ${a(b+2c+d)=c(c+d)}$. Any formula for the coefficients thus breaks into cases depending on how the terms $a(b+2c+d)$ and $c(c+d)$ compare, and indeed into further nested cases to determine later steps in the order of elimination of the columns in the elimination table.

By focusing on examples without multiple eliminations, the authors found eight examples with different sequences of elimination in codimension four. Thus, any formula for the \BS decomposition of a codimension four Koszul complex necessarily breaks into at least eight distinct cases. Some of the special cases have been explored in \cite{fanny}. Experiments in Macaulay2 have provided more than three hundred different elimination sequences without multiple eliminations in codimension five.

\section{Tensor products of diagrams}\label{sec:main}

The cone of Betti diagrams has a natural multiplication operation induced by the tensor product operation on complexes. In this section, we give a formula for decomposing the product of two diagrams $\beta$ and $\beta'$ into a sum of pure diagrams, given the corresponding decompositions for $\beta$ and $\beta'$ respectively.

\begin{defn}
 Let $\beta$ and $\beta'$ be Betti diagrams. The \defi{tensor product} of $\beta$ and $\beta'$ is the diagram $\beta \goodtimes \beta'$ defined by
\[(\beta \goodtimes \beta')_{i,j} = \sum_{\substack{i_1+i_2=i\\j_1+j_2=j}}\b_{i_1,j_1}\cdot\beta'_{i_2,j_2}.\]
\end{defn}

Note that if $C_{\bullet}$ and $C_{\bullet}'$ are complexes and $\rm{Tot}(C_{\bullet}\otimes C_{\bullet}')$ is the total complex of the double complex $C_{\bullet}\otimes C_{\bullet}'$, then we have
\begin{equation}\label{eq:tensor-complex}
    \beta(\rm{Tot}(C_{\bullet}\otimes C_{\bullet}')) = \beta(C_{\bullet}) \goodtimes \beta(C_{\bullet}').
\end{equation}
Note also that the tensor product of Betti diagrams is bilinear, i.e.,
\[\beta \goodtimes (\beta'+\beta'') = \beta \goodtimes \beta'+ \beta\goodtimes \beta''\quad\rm{and}\quad(\beta'+\beta'') \goodtimes \beta = \beta' \goodtimes \beta+ \beta'' \goodtimes \beta.\]

It follows that in order to give a decomposition of $\beta \goodtimes \beta'$ as a sum of pure diagrams, it suffices to do so in the case when $\beta$ and $\beta'$ are pure. Theorem \ref{thm:shuffle} below explains how this can be done. We first introduce some notation.

Given a degree sequence $\dd=(d_0,d_1,\ldots,d_n)$, its \defi{first difference} is defined as the sequence 
\[\Delta(\dd) = (d_1-d_0,d_2-d_1,\ldots,d_n-d_{n-1}).\]
If furthermore $e$ is an integer, we let
\[\shift(\dd,e)=(e,e+d_0,e+d_0+d_1,\ldots,e+d_0+d_1+\cdots+d_n).\]
Observe that
\[\shift(\Delta(\dd),d_0)=\dd.\]

Given an ordered set (sequence) $S=(s_1,\ldots,s_r)$, we let $\defi{Perm(S)}$ denote the set of permutations of $S$. We can identify $\rm{Perm}(S)$ with the set $\mathfrak{S}_r$ of permutations of $\{1,\ldots,r\}$, by identifying a permutation $\s\in\mathfrak{S}_r$ with the sequence $(s_{\s(1)},\ldots,s_{\s(r)})$. Given disjoint ordered sets $A$ and $B$, we write $A\sqcup B$ for their concatenation. We define the \defi{shuffle product} $\defi{Sh(A,B)}$ to be the subset of $\rm{Perm}(A\sqcup B)$ that preserves the order of the elements of each of $A$ and $B$. More generally, if $A_1,\ldots,A_n$ are ordered sets, we can define \defi{Sh($A_1,\ldots,A_n$)} to be the set of order preserving permutations of $A_1\sqcup\cdots\sqcup A_n$. Note that if each $A_i=(s_i)$ is a singleton, then $\rm{Sh}(A_1,\ldots,A_n)=\rm{Perm}(S)$.

\begin{thm}\label{thm:shuffle}
 Let $\cc=(c_0,\ldots,c_m)$ and $\dd=(d_0,\ldots,d_n)$ be degree sequences. Letting $A=\Delta(\cc)$, $B=\Delta(\dd)$, we have
\begin{equation}\label{eq:prodpcpd}
\pure{\cc} \goodtimes \pure{\dd} = \sum_{\s\in\rm{Sh}(A,B)}\pure{\shift(\s,c_0+d_0)}. 
\end{equation}
More generally, for degree sequences $\dd^1,\ldots,\dd^r$, and $A_i=\Delta(\dd^i)$, we have
\[\biggoodtimes_{i=1}^r \pure{\dd^i}=\sum_{\s\in\rm{Sh}(A_1,\ldots,A_r)}\pure{\shift(\s,d^1_0+\cdots+d^r_0)}.\]
\end{thm}

\begin{example}
 Let $\cc=(0,3,5)$, $\dd=(0,1,6)$ so that $\Delta(\cc)=(3,2)$ and ${\Delta(\dd)=(1,5)}$. Then
\begin{align*}
    \pure{\cc}\goodtimes\pure{\dd} & = \pure{0,3,5,6,11} \\
         &\ \ + \pure{0,3,4,6,11} \\
         &\ \ + \pure{0,3,4,9,11} \\
         &\ \ + \pure{0,1,4,6,11} \\
         &\ \ + \pure{0,1,4,9,11} \\
         &\ \ + \pure{0,1,6,9,11}.
\end{align*}
\end{example}

Before proving Theorem \ref{thm:shuffle}, we need a preliminary lemma. In the statement of this lemma, we will write $\rm{Prod}(S)$ for the product $s_1\cdot(s_1+s_2)\cdots(s_1+s_2+\cdots+s_r)$, where $S=(s_1,\ldots,s_r)$.

\begin{lem}\label{lemma:shufprod}
 Let $A=(e_1,\ldots,e_m)$ and $B=(f_1,\ldots,f_n)$ be ordered sets. We have
\[
\begin{split}
 \sum_{\s\in\rm{Sh}(A,B)}\frac{1}{\s(1)\cdot(\s(1)+\s(2))\cdots(\s(1)+\s(2)+\cdots+\s(m+n))}=\\
\frac{1}{e_1\cdot(e_1+e_2)\cdots(e_1+\cdots+e_m)\cdot f_1\cdot(f_1+f_2)\cdots(f_1+\cdots+f_n)}.
\end{split}
\]
In terms of $\rm{Prod}(S)$, the above formula can be rewritten as
\[\sum_{\s\in\rm{Sh}(A,B)}\frac{\rm{Prod}(A)\cdot\rm{Prod}(B)}{\rm{Prod}(\s)}=1.\]
More generally, if $A_1,\ldots,A_r$ are ordered sets, then
\[\sum_{\s\in\rm{Sh}(A_1,A_2,\ldots,A_r)}\frac{\rm{Prod}(A_1)\cdot\rm{Prod}(A_2)\cdots\rm{Prod}(A_r)}{\rm{Prod}(\s)}=1.\]
\end{lem}

\begin{proof} The general statement for $A_1,\ldots,A_r$ follows from the one for $A,B$ by induction, so it suffices to treat the latter.

 Note that for any $\s\in\rm{Sh}(A,B)$ we either have $\s(m+n)=e_m$ or $\s(m+n)=f_n$. We can identify the set of $\s$ with $\s(m+n)=e_m$ with $\rm{Sh}(A\setminus e_m,B)$ and likewise, identify the set of $\s$ with $\s(m+n)=f_n$ with $\rm{Sh}(A,B\setminus f_n)$. Observe also that $\rm{Prod}(A)=\rm{Prod}(A\setminus e_m)\cdot(e_1+\cdots+e_m)$ and similarly for $\rm{Prod}(B)$ and $\rm{Prod}(\s)$. 

The preceding remarks allow us to write
\[
\begin{split}
\sum_{\s\in\rm{Sh}(A,B)}\frac{\rm{Prod}(A)\cdot\rm{Prod}(B)}{\rm{Prod}(\s)}&=\frac{\sum_i e_i}{\sum_i e_i + \sum_j f_j}\cdot\sum_{\s\in\rm{Sh}(A\setminus e_m,B)}\frac{\rm{Prod}(A\setminus e_m)\cdot\rm{Prod}(B)}{\rm{Prod}(\s)}\\
&+\frac{\sum_j f_j}{\sum_i e_i + \sum_j f_j}\cdot\sum_{\s\in\rm{Sh}(A,B\setminus f_n)}\frac{\rm{Prod}(A)\cdot\rm{Prod}(B\setminus f_n)}{\rm{Prod}(\s)}.
\end{split}
\]

By induction,
\[\sum_{\s\in\rm{Sh}(A\setminus e_m,B)}\frac{\rm{Prod}(A\setminus e_m)\cdot\rm{Prod}(B)}{\rm{Prod}(\s)}=\sum_{\s\in\rm{Sh}(A,B\setminus f_n)}\frac{\rm{Prod}(A)\cdot\rm{Prod}(B\setminus f_n)}{\rm{Prod}(\s)}=1,\]
so we get
\[\sum_{\s\in\rm{Sh}(A,B)}\frac{\rm{Prod}(A)\cdot\rm{Prod}(B)}{\rm{Prod}(\s)}=\frac{\sum_i e_i}{\sum_i e_i + \sum_j f_j}+\frac{\sum_j f_j}{\sum_i e_i + \sum_j f_j}=1. \qedhere \]
\end{proof}

\begin{proof}[Proof of Theorem~\ref{thm:shuffle}] 
The more general statement follows from the case of two degree sequences by induction, so we focus on the latter.

Shifting degrees, we may assume that $c_0=d_0=0$. We write $\Delta(\cc)=(e_1,\ldots,e_m)$ and $\Delta(\dd)=(f_1,\ldots,f_n)$. In what follows we will treat the $e_i$'s and $f_j$'s as indeterminates (in particular we think of the rows of the Betti diagrams as being indexed by polynomials in $e_i,f_j$), and the conclusion of the theorem will follow by specializing these indeterminates to integer values. It is enough to check that for any $1\leq k\leq m$ and $1\leq l\leq n$ the terms in (\ref{eq:prodpcpd}) agree in position $(k+l,c_k+d_l) = (k+l,e_1+\cdots+e_k + f_1+\cdots+f_l)$.

If we let $A(k)=(e_k,\ldots,e_1)$, $A'(k)=(e_{k+1},\ldots,e_m)$, $B(l)=(f_l,\ldots,f_1)$ and $B'(l)=(f_{l+1},\ldots,f_{n})$, then
\[\pure{\cc}_{k,c_k}=\frac{1}{\rm{Prod}(A(k))\cdot\rm{Prod}(A'(k))},\]
and
\[\pure{\dd}_{l,d_l}=\frac{1}{\rm{Prod}(B(l))\cdot\rm{Prod}(B'(l))}.\]
It follows that
\[(\pure{\cc} \goodtimes \pure{\dd)}_{k+l,c_k+d_l}=\frac{1}{\rm{Prod}(A(k))\cdot\rm{Prod}(A'(k))}\cdot\frac{1}{\rm{Prod}(B(l))\cdot\rm{Prod}(B'(l))}.\]

Turning attention to the right hand side of (\ref{eq:prodpcpd}), we observe that in order for $\pure{\shift(\s,0)}_{k+l,c_k+d_l}$ to be nonzero, we must have that $\s(k+l)=c_k+d_l$. It follows that we can identify $\s$ with a pair $(\t,\t')$, where $\t\in\rm{Sh}(A(k),B(l))$ and ${\t'\in\rm{Sh}(A'(k),B'(l))}$. It is then easy to see that
\[\pure{\shift(\s,0)}_{k+l,c_k+d_l}=\frac{1}{\rm{Prod}(\t)\cdot\rm{Prod}(\t')},\]
so the entry of the right hand side of (\ref{eq:prodpcpd}) in position $(k+l,c_k+d_l)$ is equal to
\[
\sum_{\substack{\t\in\rm{Sh}(A(k),B(l))\\\t'\in\rm{Sh}(A'(k),B'(l))}}\frac{1}{\rm{Prod}(\t)\cdot\rm{Prod}(\t')}\]
\[=\left(\sum_{\t\in\rm{Sh}(A(k),B(l))}\frac{1}{\rm{Prod}(\t)}\right)\cdot\left(\sum_{\t'\in\rm{Sh}(A'(k),B'(l))}\frac{1}{\rm{Prod}(\t')}\right)\]
\[\overset{\textrm{Lemma }\ref{lemma:shufprod}}{=}\frac{1}{\rm{Prod}(A(k))\cdot\rm{Prod}(B(l))}\cdot\frac{1}{\rm{Prod}(A'(k))\cdot\rm{Prod}(B'(l))},\]
which is what we wanted to prove. \end{proof}

\begin{rem}\label{goodchoice} In Theorem~\ref{thm:shuffle}, we see the advantage of our choice of pure diagrams: no new coefficients appear on the right-hand side.
\end{rem}

\begin{cor}\label{cor:altdecomp} 
Let $S = \kk [\lst x n]$.  Suppose that we have an $S$-module $M$ and a decomposition of its Betti table
\begin{equation*}
\beta(M) = \sum_{k} a_k \, \pure{\dd^k}
\end{equation*}
into pure diagrams.  Let $f\in S$ be a homogeneous element of degree $e$ that is regular on $M$.  Then 
\[
\beta(M/(f)) = e\sum_k a_k\left[\sum_{\sigma\in\rm{Perm}(\Delta(\dd^k)\cup e)}\pure{\shift(\sigma,d^k_0)}\right].
\]
\end{cor}
\begin{proof} Write $K_{\bullet}(f)$ for the Koszul complex on $f$. Note that $\beta(K_{\bullet}(f))=\pure{0,e}$. The hypothesis now implies that $\beta(M/(f))=\beta(M) \goodtimes \pure{0,e}$. We now apply Theorem~\ref{thm:shuffle} and bilinearity to express $\beta(M/(f))$ as a sum of pure diagrams.

We have used that $\rm{Sh}(\Delta(\dd^k),\{e\})= \rm{Perm} (\Delta (\dd^k) \cup e)$ to obtain the form above.
\end{proof}

\begin{cor}\label{cor:altdecompCI}
Suppose that $S/I$ is a complete intersection where $I=(f_1, \dots, f_n)$ and $E=(e_1, \dots, e_n)$ such that $f_i$ is of degree $e_i$.  One has the following decomposition of $\beta(S/I)$ as a nonnegative rational sum of pure diagrams: 
\[
    \beta(S/I) = e_1\cdot e_2 \cdots e_n  \sum_{\sigma \in \rm{Perm}(E)} \pure{\shift(\sigma,0)} .
\]
\end{cor} 

\begin{rem}
    Using the notation in Corollary~\ref{cor:altdecompCI}, notice that the multiplicity of $S/I$ (denoted $e(S/I)$) is given by $e(S/I) = \prod_{i=1}^n e_i$.  Hence Corollary~\ref{cor:altdecompCI} tells us that
\[
    \beta(S/I) = e(S/I)  \sum_{\sigma \in \rm{Perm}(E)} \pure{\shift(\sigma,0)} .
\]

\end{rem}

We combine Corollary~\ref{cor:altdecomp}, \ref{cor:altdecompCI}, and Proposition~\ref{prop:structure} to produce even more decompositions of Betti diagrams of complete intersections into pure diagrams.

\begin{example} 
	Let $R = \kk [x,y,u,v]$ and $I = (x^2,y^3,u^4,v^7)$.  Let $M = R/(x^2,y^3,u^4)$ and observe that $v^7$ is regular on $M$. Using Proposition~\ref{prop:structure}, we find that 
\[
	\beta(M) = 42 \pure{0,2,5,9} + 12 \pure{0,3,5,9} + 36 \pure{0,3,6,9} + 12 \pure{0,4,5,9} + 42 \pure{0,4,7,9} .
\]
Applying Corollary~\ref{cor:altdecomp}, we have 
\begin{align*}
	\beta(R/I) &= 294\pure{0,7,9,12,16} + 84 \pure{0,7,10,12,16} + 252 \pure{0,7,10,13,16} \\
			 & \ \ + 84 \pure{0,7,11,12,16} + 294 \pure{0,7,11,14,16}.
\end{align*}

	Alternatively, by Corollary~\ref{cor:altdecompCI}, we express
\[
\beta(R/I) = 168 \sum_{\sigma\in\rm{Perm}\{2,3,4,7\}} \pure{\shift(\sigma,0)} \,,
\]
a sum of 24 pure diagrams with the same coefficient.
\end{example}

\section{Acknowledgements}

We thank the MSRI and the organizers of the 2011 Summer Graduate Workshop in Commutative Algebra for making this work possible.  We would especially like to thank Daniel Erman for several useful conversations. We also thank Zach Teitler, who originally proposed the problem that motivated our work, as well as the anonymous referee for valuable insight and suggestions.  Additionally, Betti tables and other calculations in this note were inspired by many Macaulay2 \cite{M2} computations.  The interested reader should contact the authors if they would like Macaulay2 code for investigating these types of decompositions further.


\end{document}